\RequirePackage[l2tabu, orthodox]{nag}
\documentclass[11pt]{article}
\usepackage{amsmath,amsthm,amsfonts,amssymb,enumerate,calc,graphicx,float,wrapfig,caption}
\usepackage[UKenglish]{babel}
\usepackage[colorlinks=true,citecolor=black,linkcolor=black,urlcolor=blue]{hyperref}
\usepackage[lmargin=30mm,rmargin=30mm,bmargin=30mm,tmargin=30mm]{geometry}
\usepackage[numbers,sort&compress]{natbib}
\renewcommand{\baselinestretch}{1.12}
\usepackage[hang]{footmisc} 

\renewcommand{\thefootnote}{\fnsymbol{footnote}}	
\allowdisplaybreaks
\usepackage[noabbrev,capitalise]{cleveref}
\crefname{lem}{Lemma}{Lemmas}
\crefname{thm}{Theorem}{Theorems}
\crefname{prop}{Proposition}{Propositions}

\newcommand\DateFootnote{
\begingroup
\renewcommand\thefootnote{}
\footnote{\today}
\setcounter{footnote}{0}
\vspace*{-3ex}
\endgroup}

\makeatletter
\renewcommand\section{\@startsection {section}{1}{\z@}%
                                   {-3ex \@plus -1ex \@minus -.2ex}%
                                   {2ex \@plus.2ex}%
                                   {\normalfont\large\bfseries}}
\renewcommand\subsection{\@startsection{subsection}{2}{\z@}%
                                     {-2.5ex\@plus -1ex \@minus -.2ex}%
                                     {1.5ex \@plus .2ex}%
                                     {\normalfont\normalsize\bfseries}}
\renewcommand\subsubsection{\@startsection{subsubsection}{3}{\z@}%
                                     {-2ex\@plus -1ex \@minus -.2ex}%
                                     {1ex \@plus .2ex}%
                                     {\normalfont\normalsize\bfseries}}
 \renewcommand\paragraph{\@startsection{paragraph}{4}{\z@}%
                                    {1.5ex \@plus.5ex \@minus.2ex}%
                                    {-1em}%
                                    {\normalfont\normalsize\bfseries}}
\renewcommand\subparagraph{\@startsection{subparagraph}{5}{\parindent}%
                                       {1.5ex \@plus.5ex \@minus .2ex}%
                                       {-1em}%
                                      {\normalfont\normalsize\bfseries}}
\makeatother

\newcommand{\arXiv}[1]{arXiv:\,\href{http://arxiv.org/abs/#1}{#1}}
\newcommand{\msn}[1]{MR:\,\href{http://www.ams.org/mathscinet-getitem?mr=MR#1}{#1}}
\newcommand{\MSN}[2]{MR:\,\href{http://www.ams.org/mathscinet-getitem?mr=MR#1}{#1}}
\newcommand{\doi}[1]{doi:\,\href{http://dx.doi.org/#1}{#1}}

\theoremstyle{plain}
\newtheorem{thm}{Theorem}
\newtheorem{lem}[thm]{Lemma}

\theoremstyle{definition}

\newcommand{\Oh}[1]{\ensuremath{\protect\mathcal{O}(#1)}}

\renewcommand{\geq}{\geqslant}
\renewcommand{\leq}{\leqslant}

\DeclareMathOperator{\mf}{mf}

\begin{document}

{\Large\bfseries\boldmath\scshape Nonrepetitive Colourings of Graphs Excluding\\[0.6ex]  a Fixed Immersion or Topological Minor}

\DateFootnote

{\large 
Paul Wollan\,\footnotemark[2]
\quad 
David~R.~Wood\,\footnotemark[3]
}

\footnotetext[2]{Department of Computer Science, University of Rome, ``La Sapienza'', Rome, Italy  (\texttt{wollan@di.uniroma1.it}). 
Supported by the European Research Council under the European UnionÕs Seventh Framework Programme (FP7/2007-2013)/ERC Grant Agreement no. 279558.}

\footnotetext[3]{School of Mathematical Sciences, Monash University, Melbourne, Australia (\texttt{david.wood@monash.edu}). \\
Research supported by  the Australian Research Council.}

\emph{Abstract.} We prove that graphs excluding a fixed immersion have bounded nonrepetitive chromatic number. More generally, we prove that if $H$ is a fixed planar graph that has a planar embedding with all the vertices with degree at least 4 on a single face, then graphs excluding $H$ as a topological minor have bounded nonrepetitive chromatic number. This is the largest class of graphs known to have bounded nonrepetitive chromatic number.
\renewcommand{\thefootnote}{\arabic{footnote}}

\section{Introduction}

A vertex colouring of a graph is \emph{nonrepetitive} if there is no path for which the first half of the path is assigned the same sequence of colours as the second half.  More precisely, a $k$-\emph{colouring} of a graph $G$ is a function $\psi$ that assigns one of $k$ colours to each vertex of $G$.  A path $(v_1,v_2,\dots,v_{2t})$ of even order in $G$ is \emph{repetitively} coloured by $\psi$ if $\psi(v_i)=\psi(v_{t+i})$ for $i\in\{1,\dots,t\}$. A colouring $\psi$ of $G$ is \emph{nonrepetitive} if no path of $G$ of even order is repetitively coloured by $\psi$. Observe that a nonrepetitive colouring is \emph{proper}, in the sense that adjacent vertices are coloured differently. The \emph{nonrepetitive chromatic number} $\pi(G)$ is the minimum integer $k$ such that $G$ admits a nonrepetitive $k$-colouring. We only consider simple graphs with no loops or parallel edges. 

The seminal result in this area is by \citet{Thue06}, who in 1906 proved that every path is nonrepetitively 3-colourable. Thue expressed his result in terms of strings over an alphabet of three characters---\citet{AGHR-RSA02} introduced the generalisation to graphs in 2002. Nonrepetitive graph colourings have since been widely studied \citep{Gryczuk-IJMMS07,Grytczuk-DM08,CSZ,HJ-DM11,Currie-EJC02,HJSS,PZ09,BreakingRhythm,BC12,BaratWood-EJC08,BV-NonRepVertex07,BK-AC04,GKM11,BGKNP-NonRepTree-DM07,Currie-TCS05,NOW,KP-DM08,FOOZ,Pegden11,BV-NonRepEdge08,Grytczuk-EJC02,AGHR-RSA02,GPZ,JS09,DS09}.  The principle result of \citet{AGHR-RSA02} was  that graphs with maximum degree $\Delta$ are nonrepetitively $\Oh{\Delta^2}$-colourable. Several subsequent papers improved the constant \citep{Gryczuk-IJMMS07,HJ-DM11,DJKW16}. The best known bound is due to \citet{DJKW16}.

\begin{thm}[\cite{DJKW16}]
\label{NonRepDegree}
Every graph with maximum degree $\Delta$ is nonrepetitively $(1+o(1))\Delta^2$-colourable.
\end{thm}

A number of other graph classes are known to have bounded nonrepetitive chromatic number. In particular, trees are nonrepetitively 4-colourable \citep{BGKNP-NonRepTree-DM07,KP-DM08}, outerplanar graphs are nonrepetitively $12$-colourable \citep{KP-DM08,BV-NonRepVertex07}, and graphs with bounded treewidth have bounded nonrepetitive chromatic number \citep{KP-DM08,BV-NonRepVertex07}. (See \cref{TreeDecompositions} for the definition of treewidth.)\ The best known bound is due to \citet{KP-DM08}.

\begin{thm}[\cite{KP-DM08}]
\label{NonRepTreewidth}
Every graph with treewidth $k$ is nonrepetitively $4^k$-colourable.
\end{thm}

The primary contribution of this paper is to provide a qualitative generalisations of   \cref{NonRepDegree,NonRepTreewidth} via the notion of graph immersions and excluded topological minors.

A graph $G$ contains a graph $H$ as an \emph{immersion} if the vertices of $H$ can be mapped to distinct vertices of $G$, and the edges of $H$ can be mapped to pairwise edge-disjoint paths in $G$, such that each edge $vw$ of $H$ is mapped to a path in $G$ whose endpoints are the images of $v$ and $w$.  The image in $G$ of each vertex in $H$ is called a \emph{branch vertex}. Structural and colouring properties of graphs excluding a fixed immersion have been widely studied \cite{GKT15,FW16a,Wollan15,DW16,Dvorak12,DDFMMS14,LT14,AL03,MW14,DK14,DMMS13,GKT13,RS10}. We prove that graphs excluding a fixed immersion have bounded nonrepetitive chromatic number. 

\begin{thm}
\label{NonRepWeakImm}
For every graph $H$ with $t$ vertices, every graph that does not contain $H$ as an  immersion is nonrepetitively $4^{t^4+O(t^2)}$-colourable.
\end{thm}

Since a graph with maximum degree $\Delta$ contains no star with $\Delta+1$ leaves as an immersion, \cref{NonRepWeakImm} implies that graphs with bounded degree have bounded nonrepetitive chromatic number (as in \cref{NonRepDegree}). 
 
We strengthen  \cref{NonRepWeakImm} as follows (although without explicit bounds).  A graph $G$ contains a graph $H$ as a \emph{strong immersion} if  $G$ contains $H$ as an immersion, such that for each edge $vw$ of $H$, no internal vertex of the path in $G$ corresponding to $vw$ is a  branch vertex. 

\begin{thm}
\label{NonRepStrongImmersion}
For every fixed graph $H$, there exists a constant $k$, such that every graph $G$ that does not contain $H$ as a strong immersion is nonrepetitively $k$-colourable.
\end{thm}

Note that planar graphs with $n$ vertices are nonrepetitively \Oh{\log n}-colourable \citep{DFJW13}, and the same is true for graphs excluding a fixed graph as a minor or topological minor \citep{DMW13}. It is unknown whether any of these classes have bounded nonrepetitive chromatic number. Our final result shows that excluding a special type of topological minor gives bounded nonrepetitive chromatic number. 

\begin{thm}
\label{NonRepTopo}
Let $H$ be a fixed planar graph that has a planar embedding with all the vertices of $H$ with degree at least 4 on a single face. Then there  exists a constant $k$, such that every graph $G$ that does not contain $H$ as a topological minor is nonrepetitively $k$-colourable.
\end{thm}

Graphs with bounded treewidth exclude fixed walls as topological minors. Since walls are planar graphs with maximum degree 3, \cref{NonRepTopo} implies that graphs of bounded treewidth have bounded nonrepetitive chromatic number (as in \cref{NonRepTreewidth}). Similarly, for every graph $H$ with $t$ vertices, the `fat star' graph (which is the 1-subdivision of the $t$-leaf star with edge multiplicity $t$) contains $H$ as a strong immersion. Since fat stars embed in the plane with all vertices of degree at least 4 on a single face, \cref{NonRepTopo} implies that graphs excluding a fixed graph as a strong immersion have bounded nonrepetitive chromatic number (as in \cref{NonRepStrongImmersion}). In this sense, \cref{NonRepTopo} generalises all of \cref{NonRepDegree,NonRepTreewidth,NonRepWeakImm,NonRepStrongImmersion}. 

The results of this paper, in relation to the best known bounds on the nonrepetitive chromatic number, are summarised in \cref{Classes}.

\begin{figure}[!h]
\hspace*{-10mm}\includegraphics{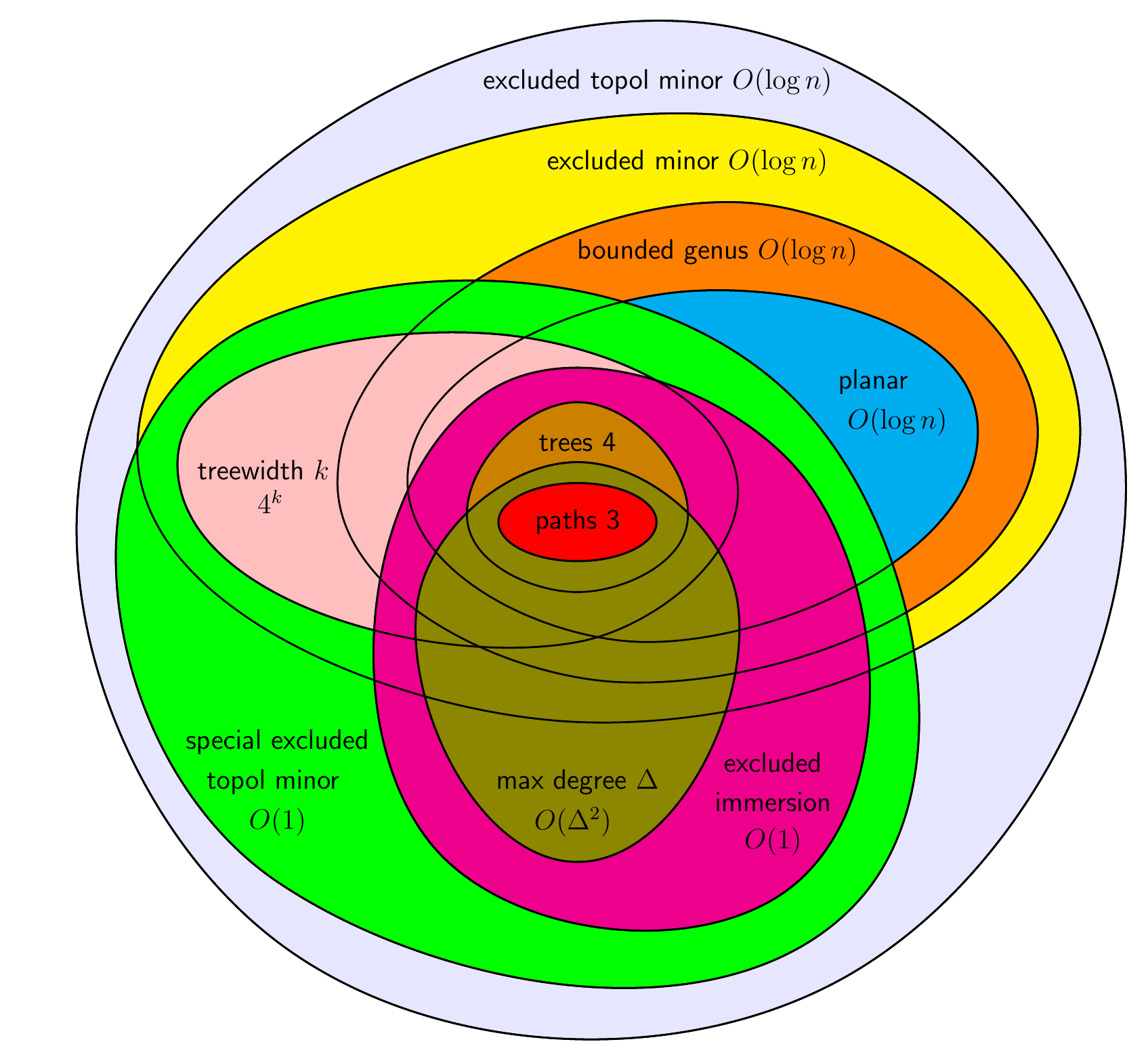}
\caption{\label{Classes} Upper bounds on the nonrepetitive chromatic number of various graph classes. `Special' refers to the condition in \cref{NonRepTopo}.}
\end{figure}

\section{Tree Decompositions}
\label{TreeDecompositions}

For a graph $G$ and tree $T$, a \emph{tree decomposition} or \emph{$T$-decomposition} of $G$ consists of a collection $(T_x\subseteq V(G):x\in V(T))$ of sets of vertices of $G$, called \emph{bags}, indexed by the nodes of $T$, such that for each vertex $v\in V(G)$ the set $\{x\in V(T):v\in T_x\}$ induces a connected subtree of $T$, and for each edge $vw$ of $G$ there is a node $x\in V(T)$ such that $v,w\in T_x$. The \emph{width} of a $T$-decomposition is the maximum, taken over the nodes $x\in V(T)$, of $|T_x|-1$. The \emph{treewidth} of a graph $G$ is the minimum width of a tree decomposition of $G$. The \emph{adhesion} of a tree decomposition $(T_x:x\in V(T))$ is $\max\{|T_x\cap T_y|:xy\in E(T)\}$. The \emph{torso} of each node $x\in V(T)$ is the graph obtained from $G[T_x]$ by adding a clique on $T_x\cap T_y$ for each edge $xy\in E(T)$ incident to $x$. \citet{DMW13} generalised  \cref{NonRepTreewidth} as follows:

\begin{lem}[\cite{DMW13}]
\label{NonRepTorso}
If a graph $G$ has a tree decomposition with adhesion $k$ such that every torso is nonrepetitively $c$-colourable, then $G$ is nonrepetitively $c\,4^k$-colourable.
\end{lem}


For integers $c,d\geq 0$ a graph $G$ has \emph{$(c,d)$-bounded degree} if $G$ contains at most $c$ vertices with degree greater than $d$. 

\begin{lem}
\label{NonRepBoundedDegree}
Every graph with $(c,d)$-bounded degree is nonrepetitively $c+(1+o(1))d^2$-colourable.
\end{lem}

\begin{proof}
Assign a distinct colour to each vertex of degree at least $d$, and colour the remaining graph by \cref{NonRepDegree}. For each vertex $v$ of degree at least $d$, no other vertex is assigned the same colour as $v$. Thus $v$ is in no repetitively coloured path. The result then follows from \cref{NonRepDegree}.
\end{proof}

\citet{Dvorak12} proved the following structure theorem for graphs excluding a strong immersion. 

\begin{thm}[\cite{Dvorak12}] 
\label{StrongImmStructure}
For every fixed graph $H$, there exists a constant $k$, such that every graph $G$ that does not contain $H$ as a strong immersion has  a tree decomposition such that each torso is $(k,k)$-bounded degree.
\end{thm}

\cref{NonRepBoundedDegree,NonRepTorso,StrongImmStructure} imply \cref{NonRepStrongImmersion}.

\section{Weak Immersions}

The proof of \cref{NonRepStrongImmersion} gives no explicit bound on the constant $k$. In this section we prove an explicit bound on the nonrepetitive chromatic number of graphs excluding a weak immersion.  \cref{NonRepWeakImm} follows from \cref{NonRepTorso} and the following structure theorem of independent interest.

\begin{thm}
\label{WeakImmStructureExplicit}
For every graph $H$ with $t$ vertices, every graph that does not contain $H$ as a weak immersion has a tree decomposition with adhesion at most $t^2$ such that every torso has $(t,t^4+2t^2)$-bounded degree.
\end{thm}

The starting point for the proof of \cref{WeakImmStructureExplicit} is the following structure theorem of \citet{Wollan15}. For a tree $T$ and graph $G$, a \emph{$T$-partition} of $G$ is a partition $(T_x\subseteq V(G):x\in V(T))$ of $V(G)$ indexed by the nodes of $T$. Each set $T_x$ is called a \emph{bag}. Note that a bag may be empty. For each edge $xy$ of a tree $T$, let $T(xy)$ and $T(yx)$ be the components of $T-xy$ where $x$ is in $T(xy)$ and $y$ is in $T(yx)$. For each edge $xy\in E(T)$, let $G(T,xy):=\bigcup\{T_z: z \in V(T(xy))\}$ and  $G(T,yx):=\bigcup\{T_z: z \in V(T(yx))\}$. Let $E(T,xy)$ be the set of edges in $G$ between $G(T,xy)$ and $G(T,yx)$. The \emph{adhesion} of a $T$-partition $(T_x:x\in V(T))$ is the maximum, taken over all edges $xy$ of $T$, of $|E(T,xy)|$. For each node $x$ of $T$, the \emph{torso} of $x$ (with respect to a $T$-partition) is the graph obtained from $G$ by identifying $G(T,yx)$ into a single vertex for each edge $xy$ incident to $x$ (deleting resulting parallel edges and loops).

\begin{thm}[\cite{Wollan15}]
\label{WeakImmStructureWollan}
For every graph $H$ with $t$ vertices, for every graph $G$ that does not contain $H$ as a weak immersion, there is a $T$-partition of $G$ with adhesion at most $t^2$ such that each torso has $(t,t^2)$-bounded degree.
\end{thm}

\begin{proof}[Proof of \cref{WeakImmStructureExplicit}]
Let $G$ be a graph that does not contain $H$ as a weak immersion. Consider the $T$-partition  $(T_x:x\in V(T))$ of $G$ from \cref{WeakImmStructureWollan}. 

Let $T'$ be obtained from $T$ by orienting each edge towards some root vertex. We now  define a tree decomposition $(T^*_x:x\in V(T))$ of $G$. Initialise $T^*_x:=T_x$ for each node $x\in V(T)$. For each edge $vw$ of $G$, if $v\in T_x$ and $w\in T_y$ and $z$ is the least common ancestor of $x$ and $y$ in $T'$, then add $v$ to $T^*_\alpha$ for each node $\alpha$ on the $\overrightarrow{xz}$ path in $T'$, and add $w$ to $T^*_\alpha$ for each node $\alpha$ on the $\overrightarrow{yz}$ path in $T'$. Thus each vertex $v\in T_x$ is in a sequence of bags that correspond to a directed path from $x$ to some ancestor of $x$ in $T'$. By construction, the endpoints of each edge are in a common bag. Thus $(T^*_x:x\in V(T))$ is a tree decomposition of $G$. 

Consider a vertex $v \in T^*_x \cap T^*_y$ for some edge $\overrightarrow{xy}$ of $T'$. Then $v$ has a neighbour $w$ in $G(T,yx)$, and $vw\in E(T,xy)$. Thus $|T^*_x \cap T^*_y| \leq  |E(T,xy)| \leq t^2$. That is, the tree decomposition $(T^*_x:x\in V(T))$ has adhesion at most $t^2$. 

Let $G^+_x$ be the torso of each node $x\in V(T)$ with respect to the tree decomposition $(T^*_x:x\in V(T))$. That is, $G^+$ is obtained from $G[T^*_x]$ by adding a clique on $T^*_x\cap T^*_y$ for each edge $xy$ of $T$. Our goal is to prove that $G^+_x$ has $(t,t^4+2t^2)$-bounded degree. 

Consider a vertex $v$ of $G^+_x$. Then $v$ is in at most one child bag $y$ of $x$, as otherwise $v$ would belong to a set of bags that do not correspond to a directed path in $T'$. Since $(T^*_x:x\in V(T))$ has adhesion at most $t^2$,  $v$ has at most $t^2$ neighbours in $T^*_x\cap T^*_p$, where $p$ is the parent of $x$ and $v$ has at most $t^2$ neighbours in $T^*_x\cap T^*_y$. Thus the degree of $v$ in $G^+_x$ is at most the degree of $v$ in $G[T^*_x]$ plus $2t^2$. Call this property ($\star$). 

First consider the case that $v\not\in T_x$. Let $z$ be the node of $T$ for which $v\in T_z$. Since $v\in T^*_x$, by construction, $x$ is an ancestor of $z$. Let $y$ be the node immediately before $x$ on the $\overrightarrow{zx}$ path in $T'$. We now bound the number of neighbours of $v$ in $T^*_x$. Say $w\in N_G(v) \cap T^*_x$. If $w$ is in $G(T,xy)$ then let $e_w$ be the edge $vw$. Otherwise, $w$ is in $G(T,yx)$ and thus $w$ has a neighbour $u$ in $G(T,xy)$ since $w\in T^*_x$; let $e_w$ be the edge $wu$. Observe that $\{e_w:w\in N_G(v)\cap T^*_x\}\subseteq E(T,xy)$, and thus $|\{e_w:w\in N_G(v)\cap T^*_x\}|\leq t^2$. Since $e_u\neq e_w$ for distinct $u,w\in N_G(v)\cap T^*_x$, we have $|N_G(v)\cap T^*_x|\leq t^2$. By ($\star$), the degree of $v$ in $G^+_x$ is at most $3t^2$. 

Now consider the case that $v\in T_x$. Suppose further that $v$ is not one of the at most $t$ vertices of degree greater than $t^2$ in the torso $Q$ of $x$ with respect to the given $T$-partition. Suppose that in $Q$, $v$ has $d_1$ neighbours in $T_x$ and $d_2$ neighbours not in $T_x$ (the identified vertices). So $d_1+d_2\leq t^2$. Consider a neighbour $w$ of $v$ in $G[T^*_x]$ with $w\not\in T_x$. Then $w\in G(T,yx)$ for some child $y$ of $x$. For at most $d_2$ children $y$ of $x$, there is a neighbour of $v$ in $G(T,yx)$. Furthermore, for each child $y$ of $x$,  $v$ has at most $t^2$ neighbours in $G(T,yx)$ since the $T$-partition has adhesion at most $t^2$. Thus $v$ has degree at most $d_1+d_2t^2\leq t^4$ in $G[T^*_x]$. By ($\star$), $v$ has degree at most $2t^2+t^4$ in $G^+_x$. 

Since $3t^2\leq t^4+2t^2$, the torso $G^+_x$ has $(t,t^4+2t^2)$-bounded degree. 
\end{proof}

\section{Excluding a Topological Minor}

\cref{NonRepTopo} is an immediate corollary of \cref{NonRepTorso} and the following structure theorem 
of \citet{Dvorak12} that extends  \cref{StrongImmStructure}.

\begin{thm}[\cite{Dvorak12}] 
\label{TopoStructure}
Let $H$ be a  fixed planar graph that has a planar embedding with all the vertices of $H$ with degree at least 4 on a single face. Then there  exists a constant $k$, such that every graph $G$ that does not contain $H$ as a topological minor has  a tree decomposition such that each torso has $(k,k)$-bounded degree.
\end{thm}

While \cref{TopoStructure} is not explcitly stated in \citep{Dvorak12}, we now explain that it is in fact a special case of Theorem~3 in \cite{Dvorak12}. This result provides a structural description of graphs excluding a given topological minor in terms of the following definition. For a graph $H$ and surface $\Sigma$, let $\mf(H, \Sigma)$ be the minimum, over all possible embeddings of $H$ in $\Sigma$, of the  minimum number of  faces such that every vertex of degree at least 4  is incident with one of these faces. By assumption, for our graph $H$ and for every surface $\Sigma$, we have $\mf(H,\Sigma)=1$. In this case, Theorem~3 of \citet{Dvorak12} says that for some integer $k=k(H)$, every graph $G$  that does not contain $H$ as a topological minor  is a clique sum of $(k,k)$-bounded degree graphs. It immediately follows that $G$ has the desired tree decomposition. See Corollary 1.4 in \cite{LT14} for a closely related structure theorem. 

The following natural open problem arises from this research: Do graphs excluding a fixed planar graph as a topological minor have bounded nonrepetitive chromatic number? And what is the structure of such graphs?

\subsection*{Acknowledgement} This research was initiated at the \emph{Workshop on New Trends in Graph Coloring} held at the Banff International Research Station in October 2016. Thanks to the organisers. 
And thanks to Chun-Hung Liu and Zden{\v{e}}k Dvo{\v{r}}{\'a}k for stimulating conversations. 


\def\soft#1{\leavevmode\setbox0=\hbox{h}\dimen7=\ht0\advance \dimen7
  by-1ex\relax\if t#1\relax\rlap{\raise.6\dimen7
  \hbox{\kern.3ex\char'47}}#1\relax\else\if T#1\relax
  \rlap{\raise.5\dimen7\hbox{\kern1.3ex\char'47}}#1\relax \else\if
  d#1\relax\rlap{\raise.5\dimen7\hbox{\kern.9ex \char'47}}#1\relax\else\if
  D#1\relax\rlap{\raise.5\dimen7 \hbox{\kern1.4ex\char'47}}#1\relax\else\if
  l#1\relax \rlap{\raise.5\dimen7\hbox{\kern.4ex\char'47}}#1\relax \else\if
  L#1\relax\rlap{\raise.5\dimen7\hbox{\kern.7ex
  \char'47}}#1\relax\else\message{accent \string\soft \space #1 not
  defined!}#1\relax\fi\fi\fi\fi\fi\fi}

\end{document}